\documentclass[hidelinks,onefignum,onetabnum]{siamart251216}

\usepackage{graphicx}%
\usepackage{amsmath,amssymb,amsfonts}%
\usepackage{tikz}
\usetikzlibrary{datavisualization}
\usetikzlibrary{datavisualization.formats.functions}
\usepackage{pgfplots}

\newsiamremark{remark}{Remark}
\newsiamremark{problem}{Problem}
\crefname{hypothesis}{Hypothesis}{Hypotheses}
\newsiamthm{claim}{Claim}
\newsiamremark{fact}{Fact}
\crefname{fact}{Fact}{Facts}

\newcommand{\dm}{\mathcal{C}}

\newcommand{\X}{\mathbf{x}}
\headers{Maximal Multiplicity Method}{M. Arnold, O. Gendelman and V. Zharnitsky}

\title{Maximal Multiplicity Method for Optimal Damping Problem. \thanks{Submitted to the editors DATE.
\funding{O.G. would like to acknowledge funding from Israel Science Foundation, grant 3085/25. M.A. was partially supported by the Simons Foundation grant MPS-TSM-00013259. }}}

\author{Maxim Dm.  Arnold\thanks{ University of Texas at Dallas, Richardson, Texas, USA
  (\email{maxim.arnold@utdallas.edu}),}
\and Oleg V. Gendelman \thanks{Technion-Israeli Institute of Technology, Haifa, Israel
  (\email{ovgend@technion.ac.il}),} 
\and Vadim Zharnitsky \thanks{University of Illinois at Urbana-Champaign, Urbana, Illinois, USA (\email{vzh@illinois.edu}).}}

\begin{document}
\maketitle

\begin{abstract}We study the problem of designing optimal damping for a system of $n$ coupled linear oscillators. Using the Weyl--Horn theorem, we construct an explicit damping matrix that realizes the optimal asymptotic decay rate of the system. We prove that this decay rate is sharp and is determined by the geometric mean of the natural frequencies. The resulting damping strategy is shown to outperform the classical Rayleigh (proportional) damping approach. We conclude by characterizing the parameter regimes in which the optimal damping matrix necessarily ceases to be positive definite and discuss the implications of this phenomenon.
\end{abstract}

\begin{keywords}
Optimal Damping, Horn-Weyl theorem, Singular Value Decomposition
\end{keywords}

\begin{MSCcodes}
15A18, 70J25, 37C75
\end{MSCcodes}

\section{Introduction} \label{sec:intro}

Consider the textbook example of a unit mass damped oscillator with one degree of freedom,
\[
\ddot x + 2c \dot x + \omega^2 x = 0,
\]
where $x=x(t)$ is a scalar function, the damping coefficient satisfies $c\ge0$, and
$\omega > 0$ is the natural frequency. 

\begin{figure}[!h]
\centering
\begin{tikzpicture}[scale=0.7,
    thick,
    spring/.style={decorate,decoration={
        coil,
        aspect=0.7,
        segment length=3mm,
        amplitude=1.5mm,
        pre length=2.5mm,
        post length=2.5mm}},
    mass/.style={
        draw,
        fill=gray!25,
        minimum width=1.1cm,
        minimum height=0.9 cm
    },
    >=stealth
]

\draw[very thick] (0,-1.15) -- (0,1.15);
\foreach \y in {-1.05,-0.75,...,1.05}
    \draw (-0.25,\y-0.18) -- (0,\y);


\node[mass] (M) at (5,0) {};


\draw[spring] (0,0.38) -- (M.west |- 0,0.38);

\draw (0,-0.38) -- (1.35,-0.38);

\draw (1.35,-0.62) -- (1.35,-0.14);
\draw (1.35,-0.62) -- (2.25,-0.62);
\draw (1.35,-0.14) -- (2.25,-0.14);

\draw (1.65,-0.38) -- (2.65,-0.38);
\draw (1.65,-0.58) -- (1.65,-0.18);

\draw (2.65,-0.38) -- (M.west |- 0,-0.38);


\end{tikzpicture}
\caption{Spring-mass system with one degree of freedom.}
\end{figure}
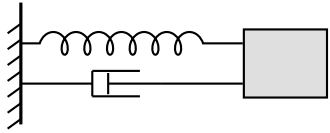

Looking for solutions of the form
$x(t)=e^{rt}$ leads to the characteristic equation
\[
r^2+2cr+\omega^2=0.
\]

The two characteristic exponents $r_1$ and $r_2$ satisfy
\(
r_1r_2=\omega^2.
\)
Since their product is fixed, the real parts of the exponents cannot both lie to the left of $-\omega$ in the complex plane. Also, the exponents cannot be of the form $r_{1,2} = -\omega \pm i b$ for $b\neq 0$. Hence, the fastest possible decay rate is attained by choosing
 $r_1=r_2 =-\omega$ if it is possible to find a corresponding damping $c$.  In this one-dimensional case it is clearly possible to achieve  by taking 
\(
c=\omega .
\)

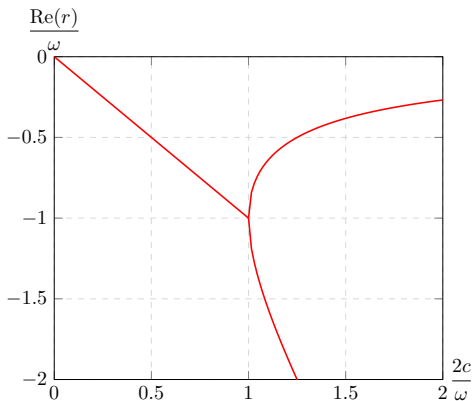
\begin{figure}[!h]
    \centering
   \begin{tikzpicture}[scale=0.75]
\begin{axis}[
    axis lines = box,
    xlabel = {$\dfrac{2 c}{\omega} $},
    ylabel = {$\dfrac{\text{Re}(r)}{\omega}$},
    xmin = 0, xmax = 2.0,
    ymin = -2.0, ymax = 0,
    domain = 0:2,
    samples = 200,
    grid = both,
    grid style = {dashed, gray!30},
    every axis x label/.style={at={(ticklabel* cs:1)},anchor=west},
    every axis y label/.style={at={(ticklabel* cs:1)},anchor=south},
]

\addplot [red, thick, domain=0:1] {-x};
\addplot [red, thick, domain=1:4, forget plot] {(-x + sqrt(x^2 - 1))/1};
\addplot [red, thick, domain=1:4] {(-x - sqrt(x^2 - 1))/1};
\end{axis}
\end{tikzpicture}
    \caption{Real parts of the characteristic exponents as a function of damping parameter, normalized by the frequency.}
    \label{fig:intro}
\end{figure}

Then  the characteristic polynomial factors as
$$
r^2+2\omega r+\omega^2=(r+\omega)^2
$$
and every solution decays at the maximal exponential rate $e^{-\omega t}$. By contrast, for any other choice $c\neq\omega$, the quadratic formula shows that at least one characteristic exponent satisfies
\[
\Re(r)>-\omega,
\]
so the asymptotic decay is necessarily slower. Thus the critically damped case $c=\omega$ is optimal. 

Despite the importance of optimal damping in mechanical structures there appear no systematic results on optimality in the systems of higher degrees of freedom. Our goal is fill this gap by providing the most general optimality condition.

In the present paper we investigate a higher-dimensional analogue of the optimization problem discussed above. More precisely, we consider an $n$-degree-of-freedom mechanical system consisting of a network of masses connected by linear springs and dampers. Its dynamics is governed by the system
\begin{equation}\label{eq:main1}
M \ddot{\X}+2C\dot{\X}+K\X=0,
\end{equation}
where $\X\in\mathbb{R}^n$ is the displacement vector, $M$ is an $n\times n$ symmetric positive definite mass matrix, $K$ is an $n\times n$ positive definite stiffness matrix, and $C$ is an $n\times n$ symmetric damping matrix, which is not assumed to be positive definite.

A natural question, of both theoretical and practical significance, is the following.

\begin{problem}
Given $M$ and $K$, determine the damping matrix $C$ that yields the fastest decay of vibrations.
\end{problem}

Since the system \eqref{eq:main1} has $2n$ characteristic exponents 
\[
\lambda_1,\lambda_2,\ldots,\lambda_{2n},
\]
the asymptotic decay rate is governed by the largest of their real parts. Consequently, the above optimization problem can be formulated as the minimax problem
\[
r_{\mathrm{cr}}
=
\min_C
\,
\max_{1\le j\le 2n}
\Re(\lambda_j).
\]
 Since for any positive definite $C$, all the eigenvalues have negative real parts, it follows that $r_{\mathrm{cr}} < 0$.

Our main result asserts that
\begin{equation}
\label{eq:main*}
r_{\mathrm{cr}}
=
-
\sqrt[2n]{\dfrac{\det(K)}{\det(M)}} 
\end{equation}
and this bound is sharp.

In Section~\ref{sec:prelim}, we reformulate equation~\eqref{eq:main1} in a convenient form, introduce the necessary notation, and establish a simple a priori lower bound for $r_{\mathrm{cr}}$. We then consider in detail the first non-trivial example of the system \eqref{eq:main1} with two degrees of freedom.   

Section~\ref{sec:algorithm} presents an explicit construction of a damping matrix $C$ that attains the optimal value of the minimax problem.

In Section~\ref{sec:examples}, we illustrate the method by working out the cases of systems with two and three degrees of freedom.

Finally, Section~\ref{sec:applications} discusses several practical applications of the proposed approach and highlights some intriguing structural properties of the optimal damping matrix $C$.

\section{Preliminaries} \label{sec:prelim}

To simplify the analysis, we begin by reducing the system \eqref{eq:main1} to a canonical form. A standard result on the simultaneous diagonalization of quadratic forms allows us to perform a linear change of coordinates under which the mass matrix $M$ becomes the identity matrix, while the stiffness matrix $K$ becomes diagonal (denoted for convenience by $\Omega^2$), with entries
\[
\omega_1^2 \geq \omega_2^2 \geq \cdots \geq \omega_n^2.
\]
For notational convenience, we continue to denote the transformed variables by the same symbols. The equation of motion therefore takes the form
\begin{equation}\label{eq:main}
I \ddot \X + 2\dm\dot \X + \Omega^2\X = 0,
\end{equation}
where $I$ is the identity matrix,  $\Omega^2$ is diagonal with $\Omega^2 = {\rm diag}[\omega_1^2, \omega_2^2, ..., \omega_n^2]$ and $\dm$ is an arbitrary symmetric matrix\footnote{not necessarily positive definite}.

\begin{lemma}
Let
$
\omega_*:= \sqrt[2n]{\det(\Omega^2)} =
\left(
\omega_1\omega_2\cdots\omega_n
\right)^{1/n}
$.
Under the above normalization, the optimal asymptotic decay rate satisfies
\[
r_{\mathrm{cr}}
\ge
-\omega_*.
\]
\end{lemma}

\begin{proof}
Seeking solutions of the form
$
\X(t)=e^{rt}\mathbf{v}
$, 
where $\mathbf{v}\neq0$ is a constant vector, leads to the characteristic equation
\begin{equation}
\label{eq:char}
P(r):=\det\!\left(I r^2+2\dm r+\Omega^2\right)=0.
\end{equation}
Its $2n$ roots are precisely the characteristic exponents 
$r_1,r_2,\ldots,r_{2n}$. Since the characteristic polynomial is monic of degree $2n$, Vieta's formula yields
\[
\prod_{j=1}^{2n} r_j
=
\det(\Omega^2)
=
\prod_{j=1}^n \omega_j^2
=
\omega_*^{\,2n}.
\]

Suppose, to the contrary, that $
\Re(r_j)<-\omega_*$
 for every $j$. Then $
|r_j|
\ge
|\Re(r_j)|
>
\omega_*$,
and therefore
\[
\prod_{j=1}^{2n}|r_j|
>
\omega_*^{\,2n}.
\]
On the other hand,
\[
\prod_{j=1}^{2n}|r_j|
=
\left|
\prod_{j=1}^{2n}r_j
\right|
=
\omega_*^{\,2n},
\]
a contradiction. Hence at least one root
satisfies
\[
\Re(r_j)\ge-\omega_*,
\]
which proves the lemma.
\end{proof}

Our main result (Theorem \ref{th:main}) shows that the lower bound established above is, in fact, attainable.

\begin{remark}
    For the general case, given by \eqref{eq:main1}, the corresponding treshhold is given by \eqref{eq:main*}. 
\end{remark}

The general construction of an optimal damping matrix will be given in the next section. Before presenting it, however, we examine in detail the simplest nontrivial case of two degrees of freedom. Besides illustrating the main ideas, this example reveals that there exists a unique symmetric matrix $\dm$ that realizes the optimal decay rate.

\subsection{Two degrees of freedom}

A preliminary version of this result was reported in \cite{arnold2026critical}.

For $n=2$, evaluating the determinant in \eqref{eq:char} gives
\begin{equation}
\label{eq:P}
\begin{aligned}
P(r)
={}&
r^4
+2\operatorname{tr}(\dm)r^3
+\bigl(\operatorname{tr}(\Omega^2)+4\det \dm\bigr)r^2
\\
&\quad
+\bigl(2c_{11}\omega_2^2+2c_{22}\omega_1^2\bigr)r
+\det(\Omega^2).
\end{aligned}
\end{equation}

Motivated by the one-dimensional critically damped oscillator, we seek a quadruple real root. As shown below, this choice realizes the fastest possible decay. Thus we require
\[
P(r)=(r+r_0)^4.
\]
Since
\(
P(0)=\det(\Omega^2)=\omega_1^2\omega_2^2,
\)
we necessarily have
\(
r_0=\sqrt{\omega_1\omega_2}.
\)
Comparing the coefficients of
\[
(r+r_0)^4
=
r^4+4r_0r^3+6r_0^2r^2+4r_0^3r+r_0^4
\]
with \eqref{eq:P}, we obtain the unique diagonal entries
\[
c_{11}
=
\frac{2\omega_1\sqrt{\omega_1\omega_2}}
{\omega_1+\omega_2},
\qquad
c_{22}
=
\frac{2\omega_2\sqrt{\omega_1\omega_2}}
{\omega_1+\omega_2},
\]
while the off-diagonal entry is determined up to sign:
\[
c_{12}
=
\pm
\frac{(\omega_1-\omega_2)^2}
{2(\omega_1+\omega_2)}.
\]
Hence the damping matrices producing the quadruple root
\(
r=-\sqrt{\omega_1\omega_2}
\)
are precisely
\begin{equation}
\label{eq:n2caseC}
\dm=
\begin{pmatrix}
\dfrac{2\omega_1\sqrt{\omega_1\omega_2}}
{\omega_1+\omega_2}
&
\displaystyle
\pm\dfrac{(\omega_1-\omega_2)^2}
{2(\omega_1+\omega_2)}
\\[10pt]
\displaystyle
\pm\dfrac{(\omega_1-\omega_2)^2}
{2(\omega_1+\omega_2)}
&
\dfrac{2\omega_2\sqrt{\omega_1\omega_2}}
{\omega_1+\omega_2}
\end{pmatrix}.
\end{equation}

\begin{remark}\label{rem:2dof}
The matrix \eqref{eq:n2caseC} is positive definite only for a restricted range of frequency ratios. Indeed,
\[
\det \dm >0
\quad\Longleftrightarrow\quad
3-\sqrt8
<
\frac{\omega_1}{\omega_2}
<
3+\sqrt8.
\]
At the boundary values
\[
\frac{\omega_1}{\omega_2}=3\pm\sqrt8
\]
one has $\det \dm=0$, whereas outside this interval $\det \dm<0$ and hence $\dm$ is sign indefinite.
\end{remark}

\subsection{Passive damping elements}

We now distinguish damping matrices that can be realized using only passive damping elements. In the standard mechanical interpretation, a damper connecting the $i$th mass to a fixed support contributes to the diagonal entry $C_{ii}$ of matrix $C$ from \eqref{eq:main1}, while a damper connecting the $i$th and $j$th masses contributes equally to $C_{ii}$ and $C_{jj}$ and with the opposite sign to $C_{ij}=C_{ji}$.

\begin{definition}
\label{def:passive}
A symmetric damping matrix $C$ is called \emph{passive} if it can be written in the form
\begin{equation}
\label{eq:passive}
\begin{aligned}
C_{ii}
&=
a_i+\sum_{j\ne i}b_{ij},
\\
C_{ij}
&=
-b_{ij},
\qquad i\ne j,
\end{aligned}
\end{equation}
where
\[
a_i>0,
\qquad
b_{ij}=b_{ji}\ge0.
\]
We call the corresponding system \eqref{eq:main} passively damped.
\end{definition}

Thus a passive damping matrix has nonpositive off-diagonal entries and strictly positive row sums:
\[
\sum_j C_{ij}=a_i>0.
\]
Conversely, any symmetric matrix with these two properties has the representation \eqref{eq:passive}.

Note that every passive damping matrix is positive definite. Indeed, for any
$\mathbf{x}\in \mathbb{R}^n$,
\[
{\mathcal R}(\mathbf{x})=\mathbf{x}^T C\mathbf{x}
=
\sum_{i=1}^n a_i x_i^2
+
\sum_{i<j}b_{ij}(x_i-x_j)^2.
\]
Since $a_i>0$ and $b_{ij} \geq 0$ , the right-hand side is strictly positive for every
$\mathbf{x}\ne0$\footnote{In Lagrangian mechanics context, the quadratic form ${\mathcal R}(\mathbf{x})$  is called Rayleigh dissipation function}.

The converse is false: positive definiteness alone does not imply passivity. For example, matrix
\(
\begin{pmatrix}
1&1\\
1&2
\end{pmatrix}
\)
is positive definite, but it is not passive because its off-diagonal entries are positive.

Thus for damping matrix the property of being passive is stronger condition than being positive definite. 
For instance, due to the remark \ref{rem:2dof}, while sufficiently large separation between the two natural frequencies prevents the optimal damping from being positive definite the optimal decay in this case cannot be achieved using only passive damping elements. 


\subsection{Relation to a physical system}

We now interpret the preceding two-dimensional result in terms of a mechanical system with two degrees of freedom. Consider
\begin{equation}
\label{eq:physical2dof}
\begin{aligned}
\ddot x_1
+2a_1\dot x_1
+2b_{12}(\dot x_1-\dot x_2)
+\omega_1^2x_1
&=0,
\\
\ddot x_2
+2a_2\dot x_2
+2b_{12}(\dot x_2-\dot x_1)
+\omega_2^2x_2
&=0.
\end{aligned}
\end{equation}
Here $a_1$ and $a_2$ describe damping of the individual oscillators, while $b_{12}$ describes the damping element coupling them.

In matrix form, the damping matrix is
\[
C=
\begin{pmatrix}
a_1+b_{12} & -b_{12}\\
-b_{12} & a_2+b_{12}
\end{pmatrix}.
\]
Hence
\begin{equation}
\label{eq:physicalC}
a_1=c_{11}+c_{12},
\qquad
a_2=c_{22}+c_{12},
\qquad
b_{12}=-c_{12}.
\end{equation}

For the optimal matrix \eqref{eq:n2caseC}, passivity requires $c_{12}<0$. We therefore choose
\[
b_{12}=-c_{12}
=
\frac{(\omega_1-\omega_2)^2}
{2(\omega_1+\omega_2)}
\ge0.
\]
The remaining damping coefficients are then
\[
a_1
=
\frac{
4\omega_1\sqrt{\omega_1\omega_2}
-(\omega_1-\omega_2)^2
}
{2(\omega_1+\omega_2)}
\]
and
\[
a_2
=
\frac{
4\omega_2\sqrt{\omega_1\omega_2}
-(\omega_1-\omega_2)^2
}
{2(\omega_1+\omega_2)}.
\]

Assume, without loss of generality, that
\[
\omega_1\ge\omega_2>0.
\]
Then $a_1\ge a_2$, and therefore the optimal damping is passive precisely when
\(
a_2>0,
\)
or
\(
4\omega_2\sqrt{\omega_1\omega_2}
>
(\omega_1-\omega_2)^2.
\)
Introducing the frequency ratio
\(
\rho=\dfrac{\omega_2}{\omega_1}\in(0,1],
\)
this condition becomes
\[
16\rho^3>(1-\rho)^4.
\]
The equation
\(
16\rho^3=(1-\rho)^4
\)
has a unique solution in $(0,1)$,
\(
\rho_* \approx 0.2638.
\)
Consequently, the optimal decay rate can be realized by passive damping elements if and only if
\[
\frac{\omega_2}{\omega_1}>\rho_*.
\]
At $\omega_2/\omega_1=\rho_*$ one has $a_2=0$, while for smaller frequency ratios the optimal damping necessarily requires an active damping element.

\subsection{Comparison with Rayleigh damping}
For systems with two or more degrees of freedom, a common practical approach is to construct damping matrix that is diagonal in the modal representation. This approach is usually called \emph{proportional} (or \emph{Rayleigh}) damping, in which $C$ is chosen from the class of symmetric matrices commuting with $K$.  The two matrices can then be simultaneously diagonalized, reducing the problem to independent one-dimensional oscillators. Optimizing each mode separately, as explained in the Introduction,  yields
\[
r_{\mathrm{cr}}
=
\max(-\omega_1,-\omega_2)
=
-\omega_2,
\]
since $\omega_1\ge\omega_2$.

\begin{remark}
The above comparison illustrates the advantage of allowing non-proportional damping. In a one-degree-of-freedom system, the optimal decay rate is determined entirely by the stiffness of the spring: the larger the stiffness, the faster the convergence to equilibrium under optimal damping.

For a two-degree-of-freedom system with, for example,
\[
\omega_1=100,
\qquad
\omega_2=1,
\]
the proportional damping approach limits the decay rate of the entire system to
\[
r_{\mathrm{cr}}=-\omega_2=-1,
\]
because the slowest mode dominates the asymptotic behavior. In contrast, by introducing coupling through the damping matrix, one can effectively redistribute the influence of the two stiffnesses across the system. Our construction shows that the optimal decay rate becomes
\[
r_{\mathrm{cr}}
=
-\sqrt{\omega_1\omega_2}
=
-10,
\]
which is an order of magnitude faster than the proportional damping limit.
\end{remark}

\section{Maximal Multiplicity Method}
\label{sec:algorithm}

We now describe general construction of an optimal damping matrix.
As before, we assume without loss of generality that
\(
\Omega^2~
=~
\operatorname{diag}
(\omega_1^2,\ldots,\omega_n^2)\),
\(
\omega_1\ge\omega_2\ge\cdots\ge\omega_n>0,
\)
and normalize so that
\(
\prod \omega_i=1.
\) 

The main idea is to factor the stiffness matrix as
\begin{equation}
\label{eq:omega}
\Omega^2=AA^T
\end{equation}
and set
\begin{equation}
\label{eq:damping}
\dm=\frac12(A+A^T).
\end{equation}

Thus, if all eigenvalues of $A$ are equal to $1$, then
\[
P(r)
=\det(Ir^2+2\dm r+\Omega^2)
=
\det(I r+A)\det(I r+A^T)
=
(r+1)^{2n}.
\]
It remains to construct a real matrix $A$ satisfying
\(
AA^T=\Omega^2
\)
and having the single eigenvalue $1$. Suppose that a real matrix $B$ has singleton spectrum $\{1\}$
and singular values
\(
\omega_1,\ldots,\omega_n.
\)
Then $BB^T$ and $\Omega^2$ are symmetric positive definite matrices with the same spectrum. Hence there exists an orthogonal matrix $U$ such that
\(
UBB^TU^T=\Omega^2.
\)
Setting
\(
A=UBU^T,
\)
the corresponding damping matrix is given by
\begin{equation}
\label{eq:dampingU}
\dm=\frac12U(B+B^T)U^T.
\end{equation}

Thus the problem reduces to constructing a real matrix with prescribed singular values
\(
\omega_1,\ldots,\omega_n
\)
and all eigenvalues equal to $1$. For this purpose we recall the theorem of Weyl and Horn \cite{horn}.

\begin{definition}
A pair of $n$-tuples
\(
(\lambda_1,\ldots,\lambda_n)\),
and \(
(\alpha_1,\ldots,\alpha_n),
\)
where $\lambda_i\in\mathbb C$ and $\alpha_i\ge0$, is called \emph{allowable} if there exists a matrix whose eigenvalues are $\lambda_i$ and whose singular values are $\alpha_i$.
\end{definition}

\begin{theorem}[Horn]
Assume that
\(
|\lambda_1|\ge\cdots\ge|\lambda_n|\),
\(
\alpha_1\ge\cdots\ge\alpha_n\ge0.
\)
The pair
\(
(\lambda_1,\ldots,\lambda_n),
\) \(
(\alpha_1,\ldots,\alpha_n)
\)
is allowable if
\[
\alpha_1\cdots\alpha_k
\ge
|\lambda_1\cdots\lambda_k|,
\qquad
k=1,\ldots,n,
\]
with equality for $k=n$.
\end{theorem}

We apply Horn's theorem with
$
\lambda_1=\cdots=\lambda_n=1
$ and $
\alpha_i=\omega_i.
$ The required inequalities are
\(
\omega_1\cdots\omega_k\ge1
\)
for \(k=1,\ldots,n,
\)
with equality for $k=n$. They follow immediately from
\(
\omega_1\ge\cdots\ge\omega_n>0
\), \(
\prod \omega_i=1.
\)
Hence a matrix $B$ with the required eigenvalues and singular values exists.

For an explicit construction we use the algorithm of Kosowski and Smoktunowicz \cite{smok}. It produces a unit lower triangular matrix $B$ with prescribed singular values and is based entirely on transformations of $2\times2$ blocks.

First, reorder the singular values
$\omega_1, \ldots \omega_n$, calling the new ordering,
\(
s_1,\ldots,s_n
\)
so that, for every $i=2,\ldots,n$, the partial product
\(
p_{i-1}=s_1s_2\cdots s_{i-1}
\)
and $s_i$ lie on opposite sides of $1$:
\begin{equation}
\label{eq:KSorder}
p_{i-1}\ge1\ge s_i
\qquad\text{or}\qquad
s_i\ge1\ge p_{i-1}.
\end{equation}
Such an ordering exists under the Weyl--Horn conditions.

For positive numbers $a,b$ satisfying
\[
a\ge1\ge b
\qquad\text{or}\qquad
b\ge1\ge a,
\]
define
\begin{equation}
\label{eq:Lblock}
L(a,b)
=
\begin{pmatrix}
1&0\\[2mm]
\sqrt{(a^2-1)(1-b^2)}&ab
\end{pmatrix}.
\end{equation}
The singular values of $L(a,b)$ are $a$ and $b$.

We now construct a sequence of $n\times n$ lower triangular matrices
\(
B_1,B_2,\ldots,B_n,
\)
all having the same singular values $s_1,\ldots,s_n$.

Start with
\(
B_1=\operatorname{diag}(s_1,\ldots,s_n).
\)
At the first step replace the leading block
\(
\operatorname{diag}(s_1,s_2)
\)
by
\(
L(s_1,s_2).
\)
The resulting matrix $B_2$ is lower triangular and has diagonal
\(
(1,\, s_1s_2,\, s_3,\ldots,s_n).
\)

Suppose inductively that $B_{i-1}$ has diagonal
\[
\underbrace{1,\ldots,1}_{i-2},
\quad
p_{i-1},
\quad
s_i,\ldots,s_n,
\]
where
\(
p_{i-1}=s_1\cdots s_{i-1}.
\)
Consider the $2\times2$ matrix
\(
L_i=L(p_{i-1},s_i).
\)
By \eqref{eq:KSorder}, this matrix is well defined and has singular values
$p_{i-1}$ and $s_i$. Let
\[
L_i
=
U_i
\begin{pmatrix}
p_{i-1}&0\\
0&s_i
\end{pmatrix}
V_i^T
\]
be its singular value decomposition, where $U_i$ and $V_i$ are orthogonal.

Embed $U_i$ and $V_i$ into the $(i-1,i)$ coordinate plane:
\[
Q_i
=
\operatorname{diag}(I_{i-2},U_i,I_{n-i}),
\qquad
Z_i
=
\operatorname{diag}(I_{i-2},V_i,I_{n-i}),
\]
and set
\begin{equation}
\label{eq:KSstep}
B_i=Q_iB_{i-1}Z_i^T.
\end{equation}
Since $Q_i$ and $Z_i$ are orthogonal, $B_i$ has the same singular values as $B_{i-1}$. Moreover, the transformation replaces the diagonal pair
\(
(p_{i-1},d_i)
\)
by the block $L_i$, whose diagonal entries are
$1$
and \(p_{i-1}s_i=p_i.
\)
Consequently, $B_i$ is again lower triangular and has diagonal
\[
\underbrace{1,\ldots,1}_{i-1},
\quad
p_i,
\quad
s_{i+1},\ldots,s_n.
\]

Repeating this step for
\(
i=2,\ldots,n
\)
gives
\[
B_n=
\begin{pmatrix}
1&0&\cdots&0\\
*&1&\ddots&\vdots\\
\vdots&\ddots&\ddots&0\\
*&\cdots&*&p_n
\end{pmatrix}.
\]
Since
\[
p_n=s_1 s_2\cdots s_n=1,
\]
the final matrix $B=B_n$ is unit lower triangular. Hence all its eigenvalues are equal to $1$, while its singular values are \(\{s_1,\ldots,s_n\}=
\{\omega_1,\ldots,\omega_n\}.
\)

Applying the orthogonal conjugation described above produces a matrix $A$ satisfying \eqref{eq:omega}, and then \eqref{eq:dampingU} gives the required damping matrix. 

The general case, when the product $\omega_1\cdot \omega_2 \cdots \omega_n$ is not necessarily equal to 1, can be deduced from the above algorithm rescaling the time variable by $\det(\Omega^2)$.
Hence, we have obtained a constructive proof of the following theorem.
\begin{theorem}
\label{th:main}
For every positive definite stiffness matrix $\Omega^2$, there is a damping matrix $\dm$, such that  the system \eqref{eq:main} attains optimal damping rate 
\[
r_{cr} = -(\omega_1 \cdots \omega_n)^{1/n}
\]
\end{theorem}

\begin{corollary}
    For the original system \eqref{eq:main1}, the critical damping exponent is given by \eqref{eq:main*}.
\end{corollary}
\begin{proof}
By multiplying with $M^{-1/2}$ on the left and by applying the transformation ${\bf x} = M^{-1/2} {\bf y}$, we obtain the equation
\[
I \ddot {\bf y} + 2 \dm {\bf y} + (M^{-1/2} K M^{-1/2}) {\bf y} =0.
\]
Now, applying the theorem we obtain that the optimal decay rate is given by 
\[
r_{cr} = - 
\sqrt[2n]{\det\!\left(M^{-\frac 1 2}\, K \, M^{- \frac 1 2}\right)} = -
\sqrt[2n]{\dfrac{\det(K)}{\det(M)}} .
\]

\end{proof}

\section{Examples}
\label{sec:examples}
Here we discuss applications of our construction to  low dimensional cases of the system \eqref{eq:main}. For the sake of brevity,  we will assume without loss of generailty \(\omega_1\ge\omega_2\ge\cdots\ge\omega_n>0,
\)
and
\(
\prod \omega_i=1.
\)
\subsection{Two degrees of freedom}

We first apply the construction of Section~\ref{sec:algorithm} to the two-dimensional case, where the optimal damping matrix is already known explicitly. 
The maximal multiplicity method constructs unit triangular matrix
\(
B=
\begin{pmatrix}
1&0\\
b&1
\end{pmatrix}
\)
with singular values $\omega_1$ and $\omega_2$. Hence we obtain
\(
b
=
\sqrt{\omega_1^2+\omega_2^2-2}
=
\omega_1-\omega_2.
\)
Thus
\[
BB^T
=
\begin{pmatrix}
\omega_1^2+\omega_2^2-1&\omega_1-\omega_2\\
\omega_1-\omega_2&1
\end{pmatrix}.
\]

An orthogonal matrix diagonalizing $BB^T$ is
\[
U
=
\frac{1}{\sqrt{1+\omega_1^2}}
\begin{pmatrix}
\omega_1&1\\
-1&\omega_1
\end{pmatrix}.
\]
Therefore
\[
\dm
=
\frac12U(B+B^T)U^T
=
\frac{1}{\omega_1+\omega_2}
\begin{pmatrix}
2\omega_1&
\dfrac12(\omega_1-\omega_2)^2\\[4pt]
\dfrac12(\omega_1-\omega_2)^2&
2\omega_2
\end{pmatrix}.
\]

This coincides with the damping matrix obtained in Section~\ref{sec:algorithm}, after imposing the normalization
\(
\omega_1\omega_2=1.
\)

\begin{remark}
In the two-dimensional case the optimal damping matrix is unique up to the sign of its off-diagonal entry. Although the Horn construction involves orthogonal transformations, the conditions
\[
\operatorname{spec}(A)=\{1\},
\qquad
AA^T=\operatorname{diag}(\omega_1^2,\omega_2^2)
\]
determine $A+A^T$, and hence $C$, uniquely up to this sign change.
\end{remark}

\subsection{Three degrees of freedom}

To give an explicit formula, we consider the case
\[
\omega_1>\omega_2>1>\omega_3>0.
\]
The remaining orderings, as well as the limiting cases, can be treated similarly.

Applying the algorithm from the preceeding section, we obtain the unit lower triangular matrix
\[
B=
\begin{pmatrix}
1&0&0\\[10pt]
-\sqrt{\dfrac{(\omega_1^2-1)(1-\omega_3^2)}
{\omega_1^2\omega_3^2(1+\omega_2^2)}}&1&0\\[12pt]
\sqrt{\dfrac{(\omega_1^2-1)(1-\omega_3^2)}
{1+\omega_2^2}}
&
\dfrac{\omega_2^2-1}{\omega_2}
&1
\end{pmatrix}.
\]
Using our asumption
\(
\omega_3=\frac{1}{\omega_1\omega_2}
\)
and introducing
\[
P:=(\omega_1^2-1)(\omega_1^2\omega_2^2-1),
\]
this can be written more compactly as
\[
B=
\begin{pmatrix}
1&0&0\\[10pt]
-\dfrac{\sqrt P}{\omega_1\sqrt{1+\omega_2^2}}
&1&0\\[12pt]
\dfrac{\sqrt P}
{\omega_1\omega_2\sqrt{1+\omega_2^2}}
&
\dfrac{\omega_2^2-1}{\omega_2}
&1
\end{pmatrix}.
\]

A direct computation gives
\[
B^TB=
\begin{pmatrix}
\omega_1^2+\omega_3^2-\dfrac1{\omega_2^2}
&
-\dfrac{\sqrt P}
{\omega_1\omega_2^2\sqrt{1+\omega_2^2}}
&
\dfrac{\sqrt P}
{\omega_1\omega_2\sqrt{1+\omega_2^2}}
\\[14pt]
-\dfrac{\sqrt P}
{\omega_1\omega_2^2\sqrt{1+\omega_2^2}}
&
\omega_2^2-1+\dfrac1{\omega_2^2}
&
\dfrac{\omega_2^2-1}{\omega_2}
\\[14pt]
\dfrac{\sqrt P}
{\omega_1\omega_2\sqrt{1+\omega_2^2}}
&
\dfrac{\omega_2^2-1}{\omega_2}
&
1
\end{pmatrix}.
\]

Let
\(
\alpha
=
\sqrt{
\frac{\omega_1^2-1}
{\omega_1^2-\omega_3^2}},
\) 
\(
\beta
=
\sqrt{
\frac{1-\omega_3^2}
{\omega_1^2-\omega_3^2}},
\)
so that
\(
\alpha^2+\beta^2=1.
\)
An orthogonal matrix diagonalizing $B^TB$ is
\[
U=
\frac{1}{\sqrt{1+\omega_2^2}}
\begin{pmatrix}
\beta\sqrt{1+\omega_2^2}
&0&
\alpha\sqrt{1+\omega_2^2}
\\[8pt]
-\omega_2\alpha
&1&
\omega_2\beta
\\[8pt]
\alpha
&\omega_2
&-\beta
\end{pmatrix},
\]
and hence
\[
U^TB^TBU=\Omega^2.
\]

\begin{remark}
The expressions for $\alpha$ and $\beta$ appear singular as
$\omega_1\to\omega_3$. 
However, this can occur only when all three frequencies tend to $1$. The completely degenerate case
\(
\omega_1=\omega_2=\omega_3=1
\)
can be treated separately and admits proportional damping.
\end{remark}

\begin{remark}

The matrix $U$ also admits a simple factorization. Introduce angles $\phi$ and $\psi$ by
\(
\alpha=\cos\phi,
\) \(
\beta=\sin\phi,
\)
\(
\tan\psi=\omega_2.
\)
Then
\[
U=
\begin{pmatrix}
1&0&0\\
0&\cos\psi&-\sin\psi\\
0&\sin\psi&\cos\psi
\end{pmatrix}
\begin{pmatrix}
\sin\phi&0&\cos\phi\\
0&1&0\\
\cos\phi&0&-\sin\phi
\end{pmatrix}.
\]
\end{remark}
The corresponding damping matrix is
\[
\dm=\frac12U^T(B^T+B)U.
\]
Its entries simplify considerably if we introduce
\[
\delta
=
\frac{(\omega_2^2-1)^2}
{2\omega_2(1+\omega_2^2)}.
\]
Then
\[
\dm=
\begin{pmatrix}
\omega_1^2\beta^2+\dfrac{2\alpha^2}{1+\omega_2^2}
&
-\delta\alpha
&
\alpha\beta\,\dfrac{3-\omega_2^2}{1+\omega_2^2}
\\[12pt]
-\delta\alpha
&
\dfrac{2\omega_2^2}{1+\omega_2^2}
&
\delta\beta
\\[12pt]
\alpha\beta\,\dfrac{3-\omega_2^2}{1+\omega_2^2}
&
\delta\beta
&
\omega_3^2\alpha^2+\dfrac{2\beta^2}{1+\omega_2^2}
\end{pmatrix}.
\]

\subsection{Non-uniqueness}

The optimal damping matrix in the two-dimensional case is essentially unique, but this property does not persist in higher dimensions. A simple dimension count already suggests that optimal damping matrices should not be unique for $n\ge3$. We do not attempt a general classification here, but give an explicit example for three degrees of freedom. Let
\(
\Omega^2=
\operatorname{diag}\left(4,1,\frac14\right).
\)

One can verify that 
\[
\dm_1
=
\dfrac{1}{20}\begin{pmatrix}
32&0&-9\\[6pt]
0&20&0\\[6pt]
-9&0&8
\end{pmatrix}
\]
 and
\[
\dm_2
=
\frac1{14}
\begin{pmatrix}
48&-8&-5\\
-8&24&-4\\
-5&-4&12
\end{pmatrix}
\]

are optimal damping matrices. Matrix $\dm_2$ is passive in the sense of Definition~\ref{def:passive} while $\dm_1$ contains active damping elements.

More generally, one can show that for
\(
\Omega^2=
\operatorname{diag}(a^2,1,a^{-2}),
\)
an optimal damping matrix can be positive definite only if
\[
a+a^{-1}<4.
\]
Since every passive damping matrix is positive definite, no passive optimal damping matrix can exist when
\(
a+a^{-1}\ge4.
\)

\section{Mechanical applications}
\label{sec:applications}

\subsection{Double car suspension problem}

At first sight, one might expect that in many mechanical applications comparable masses and spring constants lead to comparable natural frequencies, in which case proportional damping would be close to optimal. The standard double-suspension model shows that this need not be the case.

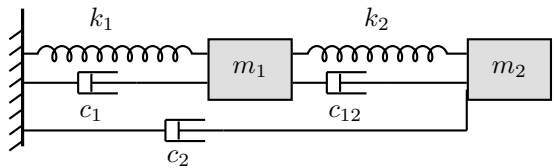
\begin{figure}[!hbt]
\centering

\begin{tikzpicture}[scale=0.7,
thick,
spring/.style={decorate,decoration={
coil,
aspect=0.7,
segment length=2mm,
amplitude=1mm,
pre length=2mm,
post length=2mm}},
mass/.style={
draw,
fill=gray!25,
minimum width=1.1cm,
minimum height=0.8cm
},
>=stealth
]

\draw[very thick] (0,-1.45) -- (0,1.15);
\foreach \y in {-1.35,-1.05,...,1.05}
\draw (-0.25,\y-0.18) -- (0,\y);

\node[mass] (M1) at (4.3,0) {$m_1$};
\node[mass] (M2) at (9.2,0) {$m_2$};

\draw[spring] (0,0.30) -- (M1.west |- 0,0.30);
\node[above] at (1.5,0.55) {$k_1$};

\draw (0,-0.25) -- (1.05,-0.25);
\draw (1.05,-0.45) -- (1.05,-0.05);
\draw (1.05,-0.45) -- (1.80,-0.45);
\draw (1.05,-0.05) -- (1.80,-0.05);
\draw (1.30,-0.25) -- (2.15,-0.25);
\draw (1.30,-0.42) -- (1.30,-0.08);
\node[below] at (1.3,-0.5) {$c_1$};
\draw (2.15,-0.25) -- (M1.west |- 0,-0.25);

\draw[spring]
(M1.east |- 0,0.30)
--
(M2.west |- 0,0.30);
\node[above] at (6.7,0.55) {$k_2$};

\draw (M1.east |- 0,-0.25) -- (5.75,-0.25);
\draw (5.75,-0.45) -- (5.75,-0.05);
\draw (5.75,-0.45) -- (6.50,-0.45);
\draw (5.75,-0.05) -- (6.50,-0.05);
\draw (6.00,-0.25) -- (6.85,-0.25);
\draw (6.00,-0.42) -- (6.00,-0.08);
\node[below] at (6.1,-0.5) {$c_{12}$};
\draw (6.85,-0.25) -- (M2.west |- 0,-0.25);

\draw (0,-1.15) -- (2.70,-1.15);
\draw (2.70,-1.35) -- (2.70,-0.95);
\draw (2.70,-1.35) -- (3.45,-1.35);
\draw (2.70,-0.95) -- (3.45,-0.95);
\draw (2.95,-1.15) -- (3.80,-1.15);
\draw (2.95,-1.32) -- (2.95,-0.98);
\node[below] at (2.95,-1.35) {$c_2$};

\draw (3.80,-1.15) -- (M2.west|-0,-1.15);
\draw (M2.west|-0,-1.15) -- (M2.west|-0,-0.38);

\end{tikzpicture}

\caption{Double-suspension system with two springs and three damping elements.}
\label{fig:suspension}
\end{figure}

Systems of this type arise frequently in vehicle dynamics; see, for example, \cite{lozia2016,mitra2016}. A schematic diagram is shown in Figure~\ref{fig:suspension}.

The equations of motion are
\begin{equation}
\label{eq:doublesusp}
\begin{cases}
m_1\ddot x_1
+c_1\dot x_1
+c_{12}(\dot x_1-\dot x_2)
+k_1x_1+k_2(x_1-x_2)=0,
\\[2mm]
m_2\ddot x_2
+c_2\dot x_2
+c_{12}(\dot x_2-\dot x_1)
+k_2(x_2-x_1)=0.
\end{cases}
\end{equation}

After the standard mass normalization, the system can be written with unit mass matrix. To simplify the formulas below, we assume
\(
m_1=m_2=1,
\)
without changing the names of other parameters. 
The stiffness matrix is then
\[
K=
\begin{pmatrix}
k_1+k_2&-k_2\\
-k_2&k_2
\end{pmatrix}.
\]
Its eigenvalues are
\(
\lambda_{1,2}
=
\dfrac{
k_1+2k_2
\pm
\sqrt{k_1^2+4k_2^2}
}{2}.
\)

\begin{remark}
Even when the masses and spring constants are equal, the two natural frequencies can be well separated. For example, if
\(
m_1=m_2=1,\)
 \(k_1=k_2=1,
\)
then
\[
\lambda_1=\frac{3+\sqrt5}{2}\approx2.62,
\qquad
\lambda_2=\frac{3-\sqrt5}{2}\approx0.38.
\]
Thus the corresponding frequencies are distinct, leaving room for non-proportional damping to improve the decay rate.
\end{remark}

We now diagonalize the stiffness matrix. Let
\(
\Delta=\sqrt{k_1^2+4k_2^2}
\)
and
\[
P=
\frac{1}{\sqrt{2\Delta}}
\begin{pmatrix}
\sqrt{\Delta+k_1}&\sqrt{\Delta-k_1}\\
-\sqrt{\Delta-k_1}&\sqrt{\Delta+k_1}
\end{pmatrix}.
\]
Then
\(
P^TKP
=
\operatorname{diag}(\lambda_1,\lambda_2).
\) The damping matrix in the original coordinates is
\[
C=
\begin{pmatrix}
c_1+c_{12}&-c_{12}\\
-c_{12}&c_2+c_{12}
\end{pmatrix}.
\]
In modal coordinates it becomes
\[
P^TCP
=
\frac1{\Delta}
\begin{pmatrix}
\dfrac{
a+b}{2}
&
c
\\[3mm]
c
&
\dfrac{
a-b}{2}
\end{pmatrix}.
\]
where $a=\Delta(c_1+c_2+2c_{12})
$, $b=k_1(c_1-c_2)
+4k_2c_{12}$, and $c=k_2(c_1-c_2)-k_1c_{12}$.

Although this formula is somewhat cumbersome, the essential point is simple. If the coefficients
$
c_1$, $c_2$, $c_{12}$
are allowed to take arbitrary real values, then
$C$ ranges over all symmetric $2\times2$ matrices. Since orthogonal conjugation is invertible, the same is true in modal coordinates. Thus the optimal damping matrix constructed in the previous sections can always be realized algebraically within this mechanical model. Whether it can be realized using only passive damping elements is a separate question governed by the positivity conditions discussed above.

\subsection{Chain of masses connected by identical springs}

Consider a chain of $n$ identical masses connected by identical springs with spring constant $k$, with the first and last masses attached to fixed walls. This may be viewed as a higher-dimensional analogue of the preceding system restricted to the nearest neighbor interaction.

Let each mass be equal to $m$. Allowing an arbitrary symmetric damping matrix $C$, the equations take the form
\[
m\ddot X+2C\dot X+KX=0,
\]
where
\[
K
=
k
\begin{pmatrix}
2&-1&0&\cdots&0\\
-1&2&-1&\cdots&0\\
0&-1&2&\ddots&\vdots\\
\vdots&\vdots&\ddots&\ddots&-1\\
0&0&\cdots&-1&2
\end{pmatrix}.
\]

The eigenvalues of $m^{-1}K$ are
\(
\lambda_j
=
4\frac{k}{m}
\sin^2
\frac{j\pi}{2(n+1)},\)
 Hence the natural frequencies are
\[
\omega_j
=
2\sqrt{\frac{k}{m}}
\sin
\frac{j\pi}{2(n+1)}\qquad j=1,\ldots,n.
\]

By the main theorem, the optimal decay rate is the negative geometric mean of these frequencies:
\[
r_{\rm cr}
=
-2\sqrt{\frac{k}{m}}
\left(
\prod_{j=1}^n
\sin\frac{j\pi}{2(n+1)}
\right)^{1/n}.
\]
As $n\to\infty$,
\[
r_{\rm cr}\longrightarrow
-\sqrt{\frac{k}{m}}.
\]

Thus, the limiting value
$
-\sqrt{\frac{k}{m}}
$
gives the optimal decay rate in the continuum limit of the chain.

\section{Discussion and open problems}

We have considered a classical mass--spring system with an unrestricted symmetric damping matrix. In applications, however, not all damping coefficients can necessarily be chosen independently. Once the optimal unconstrained decay rate is known, a natural next problem is to optimize the damping matrix subject to structural, sparsity, or passivity constraints. We plan to study such constrained optimization problems elsewhere.

A second question concerns uniqueness. For one and two degrees of freedom, the optimal damping matrix is essentially unique. Starting from $n=3$, this is no longer the case. As suggested by dimension count and illustrated by the examples above, there are continuous families of optimal damping matrices; in the three-dimensional case we have exhibited at least a one-parameter family. A systematic description of these families, and of their possible connected components, remains open.

Passivity imposes an additional restriction. In the two-dimensional case we obtained a complete criterion for when the optimal damping can be realized using only passive damping elements. For $n=3$, the situation is more complicated. We found explicit passive optimal damping matrices, as well as parameter ranges in which passivity is impossible, but a complete classification is still lacking. More generally, it would be useful to characterize when the optimal damping matrix can be chosen passive for a prescribed stiffness matrix.

Now, we discuss the issue  of robustness. The maximal multiplicity construction places all characteristic roots at the same point. Such a multiple root is necessarily sensitive to perturbations. A perturbation of size $\varepsilon$ typically produces a splitting on the scale $\varepsilon^{1/n}$.

The model equation
\[
(r+1)^n=0
\]
already illustrates this behavior. After a perturbation
\[
(r+1)^n=\varepsilon,
\]
one obtains
\[
r=-1+\varepsilon^{1/n}
\]
for one of the perturbed roots. Thus even a small perturbation may produce a noticeably larger displacement when the multiplicity is high.

This sensitivity is already present in the one-dimensional critically damped oscillator and becomes more pronounced as the number of degrees of freedom increases. In practical applications, it may therefore be preferable to sacrifice some decay rate in exchange for greater robustness. Determining the optimal trade-off between decay rate and spectral stability is another natural problem for further study.

Finally we address the issue of having terms of the type $t^k e^{-\omega t} $ due to the presence of multiple eigenvalues. From practical point of view such terms are undesirable as they make the solutions increase in the beginning followed by the exponential decay. This phenomenon is already observed in the one-dimensional case $n=1$. Depending on the specific application one may need to find a trade-off between rapid exponential decay  and minimizing the initial increase by splitting the eigenvalues.

\bibliographystyle{plain}
\bibliography{damping}

\end{document}